\documentclass[11pt,reqno]{amsart}
\usepackage[T1]{fontenc}
\usepackage[utf8]{inputenc}
\usepackage[english]{babel}
\usepackage{amssymb}
\usepackage{mathtools}
\usepackage[margin=2.7cm]{geometry}
\usepackage{tikz}
\tikzset{
  p2 picture/.style={every node/.style={font=\scriptsize}},
  p2 index/.style={font=\scriptsize,text=black!85},
  p2 grid/.style={draw=black!35,very thin},
  p2 interval/.style={draw=black,fill=white,thick},
  p2 path/.style={draw=black,semithick},
  p2 guide/.style={draw=black!55,densely dotted}}

\usepackage{csquotes}
\usepackage[backend=biber, style=numeric, sorting=none,
  giveninits=true, maxbibnames=99]{biblatex}
\renewbibmacro{in:}{}
\usepackage[unicode,hidelinks]{hyperref}
\hypersetup{pdftitle={Counting Lie ideals of niltriangular matrices},
  pdfauthor={N. D. Khodyunya},
  pdfkeywords={Lie algebra ideal, niltriangular matrix, enumeration of ideals, nonnesting set partition}}

\newtheorem{theorem}{Theorem}
\newtheorem{corollary}[theorem]{Corollary}
\newtheorem{lemma}[theorem]{Lemma}
\theoremstyle{definition}
\newtheorem{definition}[theorem]{Definition}

\newcommand{\coloneq}{\vcentcolon=}
\newcommand{\F}{\mathbb{F}_q}
\newcommand{\NT}[1]{\mathfrak{nt}_{#1}}
\newcommand{\qbin}[2]{\genfrac{[}{]}{0pt}{}{#1}{#2}_q}

\begin{document}

\title[Lie ideals of niltriangular matrices]{Counting Lie ideals of niltriangular matrices}
\author{N.~D.~Khodyunya}
\address{Siberian Federal University, Krasnoyarsk, Russia}
\email{nkhodyunya@gmail.com}

\begin{abstract}
We give an explicit finite sum for the number of ideals of the Lie algebra
of strictly lower triangular $n\times n$ matrices over $\mathbb F_q$,
valid for every prime power $q$. The sum runs over integer compositions,
with weights expressed using ordinary and Gaussian binomial coefficients.
A contraction bijection transforms Gagnon's configuration sum into a
weighted enumeration of nonnesting partitions, with antichains of intervals
chosen independently in each block. We derive a Stieltjes continued
fraction for the block weights. Lagrange inversion and the coefficient
formula for Stieltjes--Rogers polynomials then yield the explicit sum.
\end{abstract}

\maketitle
\raggedbottom

\section{Introduction}

For $n\ge1$ and a prime power $q$, let $P_n(q)$ denote the number of
ideals of the Lie algebra $\NT{n}(\F)$ of strictly lower
triangular $n\times n$ matrices. The enumeration of these ideals is the
type $A_{n-1}$ case of Problem~2 in~\cite{sigsam2001}.
% Russian original: Theorem 2, pp. 567--568.
Levchuk described these ideals in~\cite{vL76}.
Gagnon's classification by labelled tight splices and matroids
\cite[Theorem~4.1]{Gagnon21} gives a finite configuration sum for
$P_n(q)$~\cite[Corollary~4.2]{Gagnon21}.

We evaluate this sum using the contraction bijection of
Theorem~\ref{th:contraction}. It replaces the configurations by
nonnesting partitions together with antichains of intervals of
internal block elements, chosen independently in each block.
In \S\ref{sec:blocks}, we derive a continued fraction for the block weights
by decomposing noncrossing partitions at the block containing $1$.
The same decomposition for the weighted partitions, followed by
Lagrange inversion, reduces the count to coefficient extractions.
Flajolet's formula for Stieltjes--Rogers polynomials evaluates these
coefficients, giving the explicit sum over integer compositions in
Theorem~\ref{th:ideal-count}.

In the configuration sum, the terms with no selected gluing components
recover the count for ideals of the associative algebra of strictly
lower triangular matrices, the type $A_{n-1}$ case of Problem~1
in~\cite{sigsam2001}.

\section{Configurations and ideal counts}\label{sec:model}

A \emph{set partition} of $[n]\coloneq\{1,\dots,n\}$ is a collection
of pairwise disjoint nonempty subsets, called \emph{blocks}, whose
union is $[n]$. We take $[0]\coloneq\varnothing$.

An \emph{arc} of a set partition is a pair $(a,b)$, $a<b$, of elements
consecutive in a block. A partition is \emph{nonnesting} if it has no
arcs $(a,b),(a',b')$ with $a<a'<b'<b$ \cite[pp.~1556, 1565]{CDDSY07}.
Let $\mathrm{NN}_n$ denote the set of nonnesting partitions of $[n]$.
For a set partition $\lambda$, let $m(\lambda)$ denote its number of arcs.
An element of a block is \emph{internal} if it is neither its minimum
nor its maximum. An \emph{interval} of a block is a nonempty set of
consecutive elements in the order on the block, which need not be
consecutive integers.
For integers $a\le b$, write $[a,b]\coloneq\{a,a+1,\dots,b\}$.

For a partition $\lambda$ of $[n]$, its sets of left and right arc endpoints are
\[
 D(\lambda)\coloneq[n]\setminus\{\max C:C\in\lambda\},\qquad
 R(\lambda)\coloneq[n]\setminus\{\min C:C\in\lambda\}.
\]
Write
$\operatorname{succ}\colon D(\lambda)\to R(\lambda)$ for the map sending each element
to the next element of its block, and
$\operatorname{pred}$ for its inverse.

\begin{samepage}
We use the following endpoint description of nonnesting partitions
\cite[pp.~1565--1566]{CDDSY07}.

\begin{lemma}\label{l:ctr-model}
The map sending a nonnesting partition of $[n]$ to its sets
$D=\{d_1<\cdots<d_m\}$ and $R=\{r_1<\cdots<r_m\}$ of left and right
arc endpoints is a bijection onto the pairs of equally sized subsets of $[n]$
satisfying $d_i<r_i$ for $1\le i\le m$.
The corresponding partition has arcs $(d_i,r_i)$.
\end{lemma}
\end{samepage}

\begin{definition}\label{def:components}
A \emph{(gluing) component} of $\lambda\in\mathrm{NN}_n$ is a nonempty
proper suffix $\gamma$ of one block such that
$\gamma+1\coloneq\{l+1:l\in\gamma\}$ is a proper prefix of another block.
Write $\Gamma(\lambda)$ for the set of these components.
A \emph{configuration} on $[n]$ is a pair
$(\lambda,S)$ with $\lambda\in\mathrm{NN}_n$ and $S\subset\Gamma(\lambda)$; put
$k(S)\coloneq\sum_{\gamma\in S}|\gamma|$.
\end{definition}

A \emph{link} of $\lambda\in\mathrm{NN}_n$ is an integer
$l\in R(\lambda)$ with $l+1\in D(\lambda)$.
Each element of a component is a link. A link determines
its two blocks and at most one component: the prefix of the second
block must end at one plus the maximum of the first. Thus distinct
components have disjoint link sets. For a configuration $(\lambda,S)$,
distinct links give distinct arcs $(\operatorname{pred}(l),l)$,
so $k(S)\le m(\lambda)$.
Consequently $m(\lambda)-k(S)+|S|\ge0$, with equality precisely for
the partition into singletons and $S=\varnothing$.

Let $V_q(m)$ denote the number of subspaces of $\mathbb F_q^m$
contained in no coordinate hyperplane. By inclusion--exclusion
\cite[Lemma~2.3]{smj2023},
\begin{equation}\label{eq:V-total}
 V_q(m)=\sum_{i=0}^m(-1)^i\binom mi
                  \sum_{t=0}^{m-i}\qbin{m-i}{t},
\end{equation}
where $\qbin{a}{b}$ is a Gaussian binomial coefficient.
\begin{theorem}[Gagnon {\cite[Corollary~4.2]{Gagnon21}}, in our notation]\label{th:master}
For every $n\ge1$ and every prime power $q$, the number $P_n(q)$ of
ideals of $\NT{n}(\F)$ is
\begin{equation}\label{eq:main}
 P_n(q)=\sum_{\lambda\in\mathrm{NN}_n}
        \sum_{S\subset\Gamma(\lambda)}
        (q-1)^{k(S)}V_q\bigl(m(\lambda)-k(S)+|S|\bigr),
\end{equation}
where $V_q(m)$ is given by~\eqref{eq:V-total}.
\end{theorem}

\subsection*{Relation to Gagnon's formula}

We use the terminology for splices, rows, columns and bindings
from~\cite[\S\,3]{Gagnon21}, identifying a partition with its set of
arcs in this subsection.
Reversing the order of the standard basis identifies $\NT{n}(\F)$
with the upper triangular realization used in \cite{Gagnon21}.
For $\lambda\in\mathrm{NN}_n$ and $\gamma\in\Gamma(\lambda)$, put
\[
 \mathcal R_\gamma=
 \{(\operatorname{pred}(l),l+1):l\in\gamma\}
 \cup\{(\max\gamma,\operatorname{succ}(\max\gamma+1))\}.
\]
The successor chains of $\gamma$ and $\gamma+1$ show that
$\lambda\sqcup\mathcal R_\gamma$ satisfies (S1), (S2) and (T)
of~\cite[\S\,3]{Gagnon21}; hence it is a tight splice with one row.
The suffix and prefix conditions prevent further extension of this row.
% Gagnon: row description preceding Lemma 3.1 and its proof, p. 12.
Conversely, the row description in~\cite[\S\,3.1, Lemma~3.1 and its proof]{Gagnon21}
together with (T) shows that every row arises uniquely in this way.
Thus $\Gamma(\lambda)$ corresponds to the rows of the largest tight
splice of $\lambda$, and~\cite[Proposition~3.3]{Gagnon21} identifies its
tight splices with subsets $S\subset\Gamma(\lambda)$.
Such a splice has $k(S)$ bindings, $|S|$ rows, and $m(\lambda)-k(S)$
columns by~\mbox{\cite[(3.2)]{Gagnon21}}.
For the $q$-Stirling numbers of the second kind in the convention
of~\cite[(2.2)]{Gagnon21}, the subspace interpretation
in~\cite[Theorem~3.5]{milne1982} gives
\begin{equation}\label{eq:stirling}
 V_q(m)=\sum_{t=0}^m(q-1)^{m-t}
               \left\{\begin{matrix}m\\t\end{matrix}\right\}_q.
\end{equation}
Substituting these statistics into \cite[Corollary~4.2]{Gagnon21}
and using~\eqref{eq:stirling} gives~\eqref{eq:main}.
The discrete partition contributes $V_q(0)=1$,
accounting for the zero ideal.

\section{Contracting gluing components}\label{sec:contraction}

In this section we turn the selected gluing components into intervals
inside the blocks of a smaller partition. The choices of intervals
will then be independent from block to block, allowing us to evaluate
\eqref{eq:main} in \S\ref{sec:blocks}.

\begin{theorem}\label{th:contraction}
    For $n\ge1$ and $0\le k<n$, there is a bijection $\Theta$ between
    \begin{itemize}
        \item configurations $(\lambda,S)$ with $\lambda\in\mathrm{NN}_n$,
              $S\subset\Gamma(\lambda)$, and $\sum_{\gamma\in S}|\gamma|=k$;
        \item pairs $(\lambda',\mathcal A)$, where $\lambda'$ is a nonnesting
              partition of $[n-k]$ and $\mathcal A$ is an antichain under inclusion of
              intervals of internal elements of its blocks, with
              $\sum_{I\in\mathcal A}|I|=k$.
    \end{itemize}
    Each selected component $\gamma\in S$ is mapped to an interval of
    $|\gamma|$ elements, and the resulting partition has
    $m(\lambda')=m(\lambda)-k+|S|$ arcs.
\end{theorem}

Figure~\ref{fig:contraction} illustrates the contraction of a component
with two links.
\begin{figure}[ht]
\centering
\begin{tikzpicture}[p2 picture,x=11mm,y=8mm]
  \draw[p2 path] (1,0) to[bend left=50] (2,0);
  \draw[p2 path] (2,0) to[bend left=32] (4,0);
  \draw[p2 path] (3,0) to[bend left=32] (5,0);
  \draw[p2 path] (5,0) to[bend left=50] (6,0);
  \foreach \i in {1,...,6} {
    \fill (\i,0) circle (1.5pt);
    \node[p2 index,anchor=north] at (\i,-0.13) {$\i$};}
  \draw[black!65,rounded corners=2pt] (1.75,-0.65) rectangle (3.25,0.2);
  \draw[black!65,rounded corners=2pt] (3.75,-0.65) rectangle (5.25,0.2);
  \node[anchor=west] at (6.65,0.5) {$\lambda=\bigl\{\{1,2,4\},\{3,5,6\}\bigr\}$};
  \node[anchor=west] at (6.65,-0.35) {one component: links $2,4$};
  \draw[->,black!75] (2.5,-0.8) -- (2.5,-2.3);
  \draw[->,black!75] (4.5,-0.8) -- (4.5,-2.3);
  \node at (3.5,-1.55) {$\Theta$};
  \begin{scope}[yshift=-24mm]
    \draw[p2 path] (1,0) to[bend left=45] (2.5,0);
    \draw[p2 path] (2.5,0) to[bend left=40] (4.5,0);
    \draw[p2 path] (4.5,0) to[bend left=45] (6,0);
    \foreach \x in {1,6} {\fill (\x,0) circle (1.5pt);}
    \foreach \x in {2.5,4.5} {\fill (\x,0) circle (2.2pt);}
    \foreach \x/\lab in {1/1,2.5/2,4.5/3,6/4} {
      \node[p2 index,anchor=north] at (\x,-0.13) {$\lab$};}
    \draw[thick] (2.5,-0.72) -- (4.5,-0.72);
    \node[anchor=north] at (3.5,-0.82) {interval $\{2,3\}$};
    \node[anchor=west] at (6.65,0.45) {$\lambda'=\bigl\{\{1,2,3,4\}\bigr\}$};
    \node[anchor=west] at (6.65,-0.4) {$k=2,\quad m(\lambda)=4,\quad m(\lambda')=3$};
  \end{scope}
\end{tikzpicture}
\caption{Contracting a component of two links.
    The boxed pairs $\{2,3\}$ and $\{4,5\}$ become points $2$ and $3$,
    respectively. The two blocks merge, and the selected component becomes
    the marked interval of internal block elements.}\label{fig:contraction}
\end{figure}
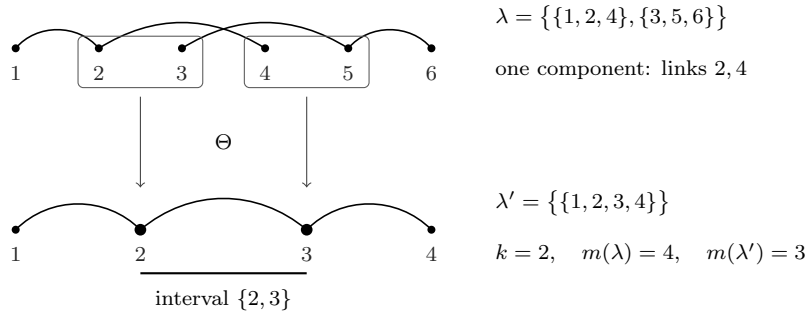

Fix a configuration $(\lambda,S)$ on $[n]$, and put
$K=\bigcup_{\gamma\in S}\gamma$ and $k=|K|=k(S)$.
Write $D=D(\lambda)$ and $R=R(\lambda)$ for the sets of left and right
arc endpoints. Identify each pair $l,l+1$ with $l\in K$, using
\begin{equation}\label{eq:contraction-map}
    \pi\colon[n]\to[n-k],\qquad \pi(i)=i-|K\cap[i-1]|,\qquad U_z=\pi^{-1}(z).
\end{equation}
The fibres $U_z$ are intervals of consecutive integers.

\begin{lemma}\label{l:ctr-arcs}
    The distinct pairs $(\pi(a),\pi(b))$, for arcs $(a,b)$ of $\lambda$,
    are the arcs of a nonnesting partition $\lambda'$ of $[n-k]$, with
    $m(\lambda')=m(\lambda)-k+|S|$.
    For each arc $z\to t$ of $\lambda'$, the arcs above it form the
    increasing matching between $D\cap U_z$ and $R\cap U_t$.
\end{lemma}

\begin{proof}
    The component definition gives, for $l\in K$,
    \[
        \begin{aligned}
            l\in D   & \ \Longrightarrow\
            \operatorname{succ}(l)\in K,\quad
            \operatorname{succ}(l+1)=\operatorname{succ}(l)+1, \\
            l+1\in R & \ \Longrightarrow\
            \operatorname{pred}(l)\in K,\quad
            \operatorname{pred}(l+1)=\operatorname{pred}(l)+1.
        \end{aligned}
    \]
    In a nontrivial fibre $U=[c,c+d]$, the selected links are
    $K\cap U=[c,c+d-1]$. Every point except possibly $c$ belongs to $D$,
    and every point except possibly $c+d$ belongs to $R$.
    Thus $D\cap U$ and $R\cap U$ are intervals of consecutive integers.
    The displayed relations show that consecutive sources in $D\cap U$
    have successors in the same fibre, and consecutive targets in
    $R\cap U$ have predecessors in the same fibre.

    Consequently, collapsing the fibres in the arc diagram and identifying
    repeated arcs gives a directed graph with at most one incoming and
    one outgoing arc at each vertex. Every arc goes to a larger vertex:
    an outgoing group includes the arc from $\max U$.
    The graph therefore defines a partition $\lambda'$.
    Each quotient arc $U\to V$ comes from the increasing bijection
    between all of $D\cap U$ and all of $R\cap V$.
    The successor map of $\lambda'$ is increasing, since the original one is
    increasing and every target fibre has a unique source fibre.
    Hence $\lambda'$ is nonnesting by Lemma~\ref{l:ctr-model}.
    For a singleton fibre on an arc, its endpoint group consists of its only point.
    The quotient sources correspond bijectively to $D\setminus K$
    via fibre maxima. Thus contraction removes $|K\cap D|=k-|S|$ arcs.
\end{proof}

\begin{lemma}\label{l:ctr-images}
    The images $\pi(\gamma)$, $\gamma\in S$, are distinct intervals of
    internal elements of blocks of $\lambda'$ and form an antichain
    under inclusion. Each component meets a fibre at most once, so
    $|\pi(\gamma)|=|\gamma|$, and
    \[
        |U_z|=1+\bigl|\{\gamma\in S:z\in\pi(\gamma)\}\bigr|.
    \]
\end{lemma}

\begin{proof}
    Each selected component follows a path of the quotient graph, whose
    arcs go to larger vertices. Its image is therefore an interval of
    the same cardinality. Its vertices are internal, because $l\in R$
    and $l+1\in D$ for each selected link $l$.
    Within a fibre, a component can end only at the smallest link and
    can begin only at the largest link. Along a quotient arc, the order
    of the continuing components is preserved by the increasing matching
    of original arcs. New components thus enter this order on the right,
    while ending components leave it on the left. Their left and right
    ends occur in the same strict order within each block.
    Hence the images are distinct and form an antichain.
    Finally, the selected links in $U_z$ are all its vertices except its
    maximum. They belong to distinct selected components, which gives
    the formula for $|U_z|$.
\end{proof}

\begin{proof}[Proof of Theorem~\ref{th:contraction}]
    Lemmas~\ref{l:ctr-arcs} and~\ref{l:ctr-images} show that contraction by
    \eqref{eq:contraction-map} defines a map from configurations to the pairs
    described in the theorem:
    \[
        \Theta(\lambda,S)\coloneq
        \bigl(\lambda',\{\pi(\gamma):\gamma\in S\}\bigr).
    \]
    This map has the required arc and component counts.

    To reconstruct a configuration, we first determine which boundary points
    of a fibre occur as arc endpoints. For an arc $z\to t$ of $\lambda'$, the
    source group $D\cap U_z$ consists of all of $U_z$, except that its minimum
    is omitted precisely when a selected component ends in $U_z$.
    The target group $R\cap U_t$ similarly omits its maximum precisely when
    a selected component starts in $U_t$.
    Indeed, the proof of Lemma~\ref{l:ctr-arcs} leaves only these two possible
    omissions. In a nontrivial fibre $[c,c+d]$, the point $c$ has no successor
    exactly when the component containing $c$ ends there; $c+d$ has no
    predecessor exactly when the component containing $c+d-1$ starts at $c+d-1$.
    A singleton fibre on an arc contributes its only point and contains no
    selected link. We now construct the inverse by expanding each vertex into
    a fibre and recovering the increasing arc matchings.

    Fix a target pair $(\lambda',\mathcal A)$ on $[n-k]$.
    For each vertex $z$, let $L_z$ be the list of intervals containing $z$,
    ordered by their left ends, and put $d_z=|L_z|$. Replace $z$ by
    \[
        U_z=[c_z,c_z+d_z],\qquad c_z=z+\sum_{v<z}d_v.
    \]
    Since $\sum_z d_z=\sum_{I\in\mathcal A}|I|=k$, these fibres partition $[n]$.
    For each arc $z\to t$ of $\lambda'$, the list of intervals continuing
    across that arc is obtained from $L_z$ by deleting the interval ending
    at $z$, if any, and from $L_t$ by deleting the interval starting at
    $t$, if any. The two lists agree. The antichain condition ensures that
    an ending interval is first and a starting interval is last, with at
    most one of each.

    Match $U_z$, omitting its minimum when an interval ends at $z$,
    to $U_t$, omitting its maximum when an interval starts at $t$,
    in increasing order. Both sets have $h+1$ points, where $h$ is the
    length of the common list. Do this for every arc of $\lambda'$.
    The groups occur in increasing order on both sides, since the successor
    map of $\lambda'$ is increasing. Thus these pairs form an increasing
    bijection between their source and target sets, with every source
    smaller than its target.
    They therefore define a nonnesting partition $\lambda$.

    Assign to $I\in\mathcal A$ the link
    $c_z+a_z(I)-1$ at each $z\in I$, where $a_z(I)$ is the position of
    $I$ in $L_z$, numbered from $1$.
    When $I$ continues across $z\to t$, its assigned link has position
    $a_z(I)-1$ in the source group if an interval ends at $z$, and
    $a_z(I)$ otherwise. This is the position of $I$ in the common
    continuing list and of its assigned link in the target group.
    Its shift by $1$ has the next position. The increasing matching
    therefore makes the assigned links and their shifts into successor chains.
    At the left end of $I$, its shift is the omitted maximum of a fibre;
    at its right end, its link is the omitted minimum.
    Each link belongs to $R(\lambda)$ and each shift to $D(\lambda)$.
    The assigned links thus form a proper suffix of a block, and their
    shifts a proper prefix of another block: they give a gluing
    component. Different intervals give different components; let $S$
    be their collection.

    \medskip

    The selected links in $U_z$ are exactly
    $c_z,\ldots,c_z+d_z-1$, so contraction restores its vertex $z$.
    The arc groups recover the arcs of $\lambda'$, and the component
    images recover $\mathcal A$.
    Conversely, for a contracted configuration Lemma~\ref{l:ctr-images}
    gives $|U_z|=d_z+1$, so expansion restores the original fibres
    and their selected links. The endpoint description established at the
    start of this proof recovers the original source and target groups;
    Lemma~\ref{l:ctr-arcs} then forces their original increasing matching.
    Thus the reconstruction recovers the original arcs and selected links.
    Since a link belongs to at most one gluing component, it also recovers
    $S$, and is inverse to $\Theta$.
\end{proof}

\section{Block weights and the enumeration formula}\label{sec:blocks}

A block of size $s\ge2$ has $s-2$ internal elements, which we identify
with $[s-2]$ in increasing order.
For $s\ge2$, let $W_s(z)$ be the generating polynomial for antichains
$\mathcal A$ under inclusion of nonempty intervals in $[s-2]$, weighted
by $z^{\sum_{I\in\mathcal A}|I|}$. Set $W_1(z)\coloneq1$, and put
\[
    A(z,y)\coloneq\sum_{s\ge1}W_s(z)y^{s-1}.
\]
A partition is \emph{noncrossing} if it has no arcs $(a,b),(c,d)$ with
$a<c<b<d$.
For $s\ge2$, replacing $[a,b]$ by $(a,b+1)$ identifies these antichains
with nonnesting partitions of $[s-1]$, preserving total length.
Scanning from left to right, match each right endpoint to the largest
unmatched left endpoint, closing before opening at each point.
This gives a noncrossing partition with the same endpoint sets.
Both matchings are unique, and total arc length depends only on the
endpoint sets. This is therefore a length-preserving bijection.
The arc lengths in each block $C$ sum to $\max C-\min C$.
Thus $A(z,y)$ enumerates noncrossing partitions, with $y$ recording their
size and each block $C$ weighted by $z^{\max C-\min C}$.
Decompose at the block containing $1$. Vertices in its inner gaps also
contribute to its span, so each inner gap contributes $A(z,zy)$; the
final gap contributes $A(z,y)$. Hence
\[
    A(z,y)=1+\frac{yA(z,y)}{1-zyA(z,zy)}
    =\frac{1}{1-\dfrac{y}{1-zyA(z,zy)}}.
\]
Iteration gives the Stieltjes continued fraction
\begin{equation}\label{eq:block-sfrac}
    A(z,y)=\cfrac{1}{1-\cfrac{y}{1-\cfrac{zy}{1-\cfrac{zy}
                {1-\cfrac{z^2y}{1-\cfrac{z^2y}{1-\ddots}}}}}}.
\end{equation}

A \emph{composition} $d=(d_1,\ldots,d_h)$ of $m$, written $d\models m$, is an
ordered tuple of positive integers with sum $m$.

\begin{theorem}\label{th:ideal-count}
    Let $n\ge1$ be an integer and $q$ a prime power. Let $V_q(m)$ count the
    subspaces of $\mathbb F_q^m$ contained in no coordinate hyperplane, as given by
    \eqref{eq:V-total}. For each $d=(d_1,\ldots,d_h)\models m$, put
    \[
        k(d)\coloneq\sum_{i=1}^h\left\lfloor\frac i2\right\rfloor d_i,
        \qquad b(d)\coloneq n-m-k(d).
    \]
    Then the number $P_n(q)$ of ideals of $\NT{n}(\F)$ is
    \begin{equation}\label{eq:composition-count}
        \begin{aligned}
            P_n(q)=1+\sum_{m=1}^{n-1}\frac{V_q(m)}{m+1}
            \sum_{\substack{d\models m\\ b(d)\ge1}}
             & \binom{n-k(d)}m\binom{b(d)+d_1-1}{d_1}                       \\
             & {}\times(q-1)^{k(d)}\prod_{i=2}^h\binom{d_{i-1}+d_i-1}{d_i}.
        \end{aligned}
    \end{equation}
    Empty sums and products are interpreted as $0$ and $1$, respectively.
\end{theorem}

\begin{proof}
    Put
    \[
        \Phi(x,y)\coloneq\frac{1}{1-xA((q-1)x,y)},
    \]
    where $A$ is given by \eqref{eq:block-sfrac}. Let $F(x,u)$ enumerate nonnesting
    partitions, including the empty partition, with each block of size $s$ weighted
    by $x^su^{s-1}W_s((q-1)x)$. The factor $x^su^{s-1}$ records its vertices and
    arcs; an interval of length $j$ contributes $(q-1)^jx^j$ through $W_s((q-1)x)$.
    Thus the power of $x$ records the number of vertices after expansion, while the
    power of $u$ records the number of arcs after contraction. Nonnesting and
    noncrossing partitions have the same distribution of block sizes \cite[solution
    to Exercise~5.44]{StanleyEC2}, so $F$ also enumerates noncrossing partitions
    with these weights. Decomposing at the block containing $1$ gives
    \[
        F=1+xF A((q-1)x,uxF),\qquad F=\Phi(x,uxF).
    \]
    Write $[x^ay^b]$ for extraction of the coefficient of $x^ay^b$. Lagrange
    inversion for $uF$ now yields
    \[
        [u^m]F=\frac{x^m}{m+1}[y^m]\Phi(x,y)^{m+1}.
    \]
    By \eqref{eq:main} and Theorem~\ref{th:contraction},
    $P_n(q)=\sum_{m\ge0}V_q(m)[x^nu^m]F$. Since $F(x,0)=(1-x)^{-1}$, the term $m=0$
    contributes $1$; all other terms have $1\le m\le n-1$. Hence
    \begin{equation}\label{eq:coefficient-count}
        P_n(q)=1+\sum_{m=1}^{n-1}\frac{V_q(m)}{m+1}
        [x^{n-m}y^m]\Phi(x,y)^{m+1}.
    \end{equation}

    We evaluate the remaining coefficients using Flajolet's formula for
    Stieltjes--Rogers polynomials \cite[Proposition~3B, p.~132]{Flajolet1980}.
    Apply it to $(1-vA(z,y))^{-1}$, whose successive continued-fraction numerators
    are $v$ and $yz^{\lfloor i/2\rfloor}$ for $i\ge1$. Taking the coefficient of
    $v^b y^m$, corresponding to compositions $(b,d_1,\ldots,d_h)$ of $b+m$, gives,
    for $b,m\ge1$,
    \begin{equation}\label{eq:rogers-blocks}
        [y^m]A(z,y)^b=
        \sum_{d\models m}z^{k(d)}\binom{b+d_1-1}{d_1}
        \prod_{i=2}^h\binom{d_{i-1}+d_i-1}{d_i}.
    \end{equation}
    On the other hand,
    \[
        \Phi(x,y)^{m+1}
        =\sum_{b\ge0}\binom{m+b}{b}x^b A((q-1)x,y)^b.
    \]
    For $m\ge1$, the term $b=0$ does not contribute. After substituting $z=(q-1)x$
    into \eqref{eq:rogers-blocks}, the term indexed by $b$ and $d$ in this
    expansion has $x$-degree $b+k(d)$. Thus $[x^{n-m}y^m]$ selects
    $b=n-m-k(d)=b(d)\ge1$. Since $\binom{m+b(d)}{b(d)}=\binom{n-k(d)}m$,
    substitution into \eqref{eq:coefficient-count} proves
    \eqref{eq:composition-count}.
\end{proof}

\begin{corollary}\label{c:polynomial}
For $n\ge5$, the number $P_n(q)$ of ideals of $\NT{n}(\F)$ is a
polynomial in $q$ of degree
$\lfloor(n-1)^2/4\rfloor$, with leading coefficient $1$ when $n$ is
odd and $2$ when $n$ is even.
\end{corollary}

\begin{proof}
Polynomiality follows from~\eqref{eq:main} and~\eqref{eq:V-total}.
For $m\ge2$, the term $i=0$ in~\eqref{eq:V-total} has degree
$\lfloor m^2/4\rfloor$, and every term with $i>0$ has smaller degree.
Thus $V_q(m)$ has this degree, with leading coefficient $1$ for even
$m$ and $2$ for odd $m\ge3$; also $V_q(1)=1$.
For a configuration $(\lambda,S)$ counted in~\eqref{eq:main}, write
$m=m(\lambda)-k(S)+|S|$ and $k=k(S)$.
If the contracted partition has $b$ blocks, Theorem~\ref{th:contraction}
gives $n=m+k+b$ with $b\ge1$.
The case $m=n-1$ forces $k=0$ and $b=1$, and hence contributes the
single term $V_q(n-1)$. If $1\le m<n-1$ and $n\ge5$, then
\[
 \deg\bigl((q-1)^kV_q(m)\bigr)
 \le n-1-m+\lfloor m^2/4\rfloor
 <\lfloor(n-1)^2/4\rfloor.
\]
The term with $m=0$ is constant. Hence $V_q(n-1)$ alone determines
the leading term.
\end{proof}

\printbibliography

\end{document}